\documentclass[12pt]{amsart}

\usepackage[latin1]{inputenc}
\usepackage{graphicx}
\usepackage{amsmath}
\usepackage{amsfonts}
\usepackage{amssymb}
\usepackage{amsthm}
\usepackage[backrefs]{amsrefs}
\usepackage{enumerate}
\usepackage{mathrsfs}
\usepackage{cite}
\usepackage[all]{xy}
\usepackage{verbatim}
\usepackage[pagebackref]{hyperref}
\usepackage[top=1.5cm,bottom=1.5cm,left=1.5cm,right=1.5cm]{geometry}
\usepackage{appendix}

\usepackage{ccfonts}
\usepackage[T1]{fontenc}

\usepackage{dsfont}

\newtheorem{theorem}{Theorem}
\newtheorem*{theorem*}{Theorem}
\newtheorem{lemma}[theorem]{Lemma}
\newtheorem*{lemma*}{Lemma}
\newtheorem{proposition}[theorem]{Proposition}
\newtheorem*{proposition*}{Proposition}

\newtheorem{corollary}[theorem]{Corollary}
\newtheorem*{corollary*}{Corollary}

\newtheorem*{definition*}{Definition}

\newtheorem*{notation*}{Notation}
\newtheorem{remark}[theorem]{Remark}
\newtheorem*{remark*}{Remark}
\newtheorem{example}[theorem]{Example}
\newtheorem*{example*}{Example}

\newtheorem*{observation*}{Observation}

\usepackage{amsmath,amssymb,setspace,nicefrac}
\usepackage[active]{srcltx}

\newcommand{\R}{{\mathbb R}}

\newcommand{\C}{{\mathbb C}}

\renewcommand{\emph}[1]{\underline{\bf{#1}}}    

\newtheorem{cor}[theorem]{Corollary}

\theoremstyle{definition}

\usepackage{color}
\usepackage{endnotes}
\newcommand{\henrik}[1]{\endnote{\textcolor{red}{#1}}}

\title{$L^2$-Betti numbers and Plancherel measure}

\author{Henrik Densing Petersen}
\address{Henrik Densing Petersen, SB-MATHGEOM-EGG, EPFL, Station 8, CH-1015, Lausanne, Switzerland}
\email{henrik.petersen@epfl.ch}

\author{Alain Valette}
\address{Alain Valette, Institut de Math{\'e}matiques, Universit{\'e} de Neuch{\^a}tel, Unimail, 11 Rue Emile Argand, CH-2000 Neuch{\^a}tel, Switzerland}
\email{alain.valette@unine.ch}

\begin{document}


\begin{abstract}
We compute $L^2$-Betti numbers of postliminal, locally compact, unimodular groups in terms of ordinary dimensions of reduced cohomology with coefficients in irreducible unitary representations and the Plancherel measure. This allows us to compute the $L^2$-Betti numbers for semi-simple Lie groups with finite center, simple algebraic groups over local fields, and automorphism groups of locally finite trees acting transitively on the boundary.
\end{abstract}

\maketitle

\section*{Introduction}
$L^2$-Betti numbers are powerful numerical invariants attached to discrete groups, and have been the subject of intense investigation (see e.g. \cite{At76,CG86,Lu02}). For a countable discrete group $\Gamma$ the $L^2$-Betti numbers $\beta^n_{(2)}(\Gamma)$ are computed as the (extended) von Neumann dimension of the right-$L\Gamma$-modules $H^n(\Gamma,\ell^2\Gamma)$. The framework for $L^2$-Betti numbers set up in \cite{Lu02} was extended by the first author \cite{Pe11} to cover locally compact, unimodular, second countable (henceforth abbreviated 'lcus') groups; it was also established that whenever such a group $G$ is totally disconnected and/or admits a cocompact lattice, then for any lattice $\Gamma$ in $G$ and all $n\geq 0$,
\begin{equation}
\beta^n_{(2)}(\Gamma) = \operatorname{covol}(\Gamma)\cdot \beta^n_{(2)}(G). \nonumber
\end{equation}
This result was extended in \cite[Theorem B]{KPV13} to cover any lattice in any lcus group.

In this article we show that whenever the locally compact group $G$ has a Plancherel measure on the unitary dual, the $L^2$-Betti numbers can be expressed explicitly in terms of this.

Recall \cite[13.9.4]{Dix69} that a separable locally compact group $G$ is {\bf postliminal} if, for every irreducible unitary representation $\pi$ of $G$, the norm closure $\overline{\pi(L^1G)}^{\lVert \cdot \rVert} \subseteq \mathcal{B}(\mathcal{H}_{\pi})$ contains the compact operators on $\mathcal{H}_{\pi}$. We abbreviate by 'plcus' the sentence 'postliminal, locally compact, unimodular, second countable'.

By \cite[Theorem 18.8.1]{Dix69} the left-regular representation of a plcus group $G$ can be completely disintegrated over the dual $\hat{G}$ in terms of irreducible unitary representations, with respect to the Plancherel measure $\mu$ on $\hat{G}$: more precisely, $L^2(G)$ decomposes as $\int_{\hat{G}}^\oplus \mathcal{H}_\omega\bar{\otimes}\mathcal{H}_{\overline{\omega}}\; d\mu(\omega)$, the left regular representation decomposes as $\int_{\hat{G}}^\oplus (\omega\otimes 1)\;d\mu(\omega)$; and the right regular representation decomposes as $\int_{\hat{G}}^\oplus (1\otimes\overline{\omega})\;d\mu(\omega)$. We denote by $H^n(G,\mathcal{H}), n\in \mathbb{N}_0$ the $n$th continuous cohomology of the locally compact group $G$ with coefficients in a continuous $G$-module $\mathcal{H}$; this carries a not necessarily Hausdorff vector topology and we denote by $\underline{H}^n(G,\mathcal{H})$ the maximal Hausdorff quotient. More detailed definitions are given below. Our main result is the following:

\begin{theorem} \label{thm:main_compute}
Let $G$ be a plcus group with a fixed Haar measure and corresponding Plancherel measure $\mu$ on the unitary dual $\hat{G}$ of $G$. Let $\hat{G}\owns \omega \mapsto \mathcal{H}_{\omega}$ be a measurable field of representatives of equivalence classes of irreducible unitary representations. Then for all $n\geq 0$
\begin{equation}
\beta^n_{(2)}(G) = \int_{\hat{G}} \dim_{\mathbb{C}} \underline{H}^n(G,\mathcal{H}_{\omega}) \mathrm{d}\mu(\omega). \nonumber
\end{equation}
\end{theorem}

Observe the subtlety that we need to take reduced cohomology. Note however, that whenever $\dim_{\mathbb{C}} H^n(G,\mathcal{H}_{\omega}) < \infty$, the cohomology is Hausdorff \cite[Proposition III.2.4]{Gu80}. 

Recall that a unitary irreducible representation $\omega$ of $G$ is {\bf square-integrable} if it appears as a sub-representation of the regular representation of $G$ on $L^2(G)$. The set of (equivalence classes of) square-integrable representations is called the {\bf discrete series} of $G$ and is denoted by $\hat{G}_d$. Recall that $\hat{G}_d$ is exactly the atomic part of the Plancherel measure $\mu$, and that for a discrete series representation $\omega$, the measure $\mu(\omega)$ is the {\bf formal dimension} of $\omega$, see \cite[Proposition 18.8.5]{Dix69}. 


In case $G$ admits a discrete series of square-integrable irreducible representations, we can refine the argument of Theorem \ref{thm:main_compute} slightly by proving that, for square-integrable irreducible representations $\omega$, the equality $\dim_{\mathbb{C}}H^n(G,\mathcal{H}_\omega)=\dim_{\mathbb{C}}\underline{H}^n(G,\mathcal{H}_\omega)$ holds; as we mentioned earlier, this implies that if $\dim_{\mathbb{C}} H^n(G,\mathcal{H}_\omega)$ is finite, then $H^n(G,\mathcal{H}_\omega)$ is reduced\footnote{For $n=1$, this follows from general principles: using the fact that $\omega$ maps the group $C^*$-algebra of $G$ to the compact operators, we deduce that $\{\omega\}$ is a closed point in $\hat{G}$, hence $H^1(G,\mathcal{H}_\omega)$ is Hausdorff, by a well-known criterion of Guichardet.}. So Theorem \ref{thm:main_compute} implies the following:

\begin{theorem} \label{thm:DS_modification}
Let $G$ be a plcus group.
\begin{equation}
\beta^n_{(2)}(G) = \left( \sum_{\omega \in \hat{G}_d} \mu(\omega) \cdot \dim_{\mathbb{C}} H^n(G,\mathcal{H}_{\omega}) \right) +  \int_{\hat{G}\backslash\hat{G}_d} \dim_{\mathbb{C}} \underline{H}^n(G,\mathcal{H}_{\omega}) \mathrm{d}\mu(\omega). \nonumber
\end{equation}
So, if $\beta^n_{(2)}(G)<\infty$, then $card\{ \omega \in \hat{G}_d \mid H^n(G,\mathcal{H}_{\omega}) \neq 0\} < \infty$, and $\mu\{ \omega \in \hat{G} \mid \underline{H}^n(G,\mathcal{H}_{\omega}) \neq 0\}<\infty$, and $\mu\{\omega\in\hat{G}: \dim\underline{H}^n(G,\mathcal{H}_\omega)=\infty\}=0$.
\end{theorem}

A sufficient condition for finiteness of $\beta^n_{(2)}(G)$, is that $G$ contains a lattice $\Gamma$ with $\beta^n_{(2)}(\Gamma)<\infty$ (this holds e.g. if $\Gamma$ admits a classifying space with finite $n$-skeleton). Also, $\beta^1_{(2)}(G) <\infty$ as soon as $G$ is compactly generated, see Proposition 4.6 of \cite{KPV13}.

From Corollary E of \cite{KPV13}, we immediately get:

\begin{corollary} \label{cor:conclusions}
Let $G$ be a plcus group. 

If $G$ admits a non-compact, closed, normal, amenable subgroup then $\underline{H}^n(G,\mathcal{H}_{\omega}) = \{ 0 \}$ for almost every $\omega\in \hat{G}$. In particular, $G$ admits no square integrable irreducible unitary representations with non-zero cohomology.

\end{corollary}

This applies in particular to unimodular, real algebraic groups with non-compact radical (these are plcus, by a result of Dixmier \cite{Dix57}), and to reductive $p$-adic algebraic groups with non-compact center (those are plcus, by a result of Bernstein \cite{Ber74}).

Theorem \ref{thm:DS_modification} allows us to compute the $L^2$-Betti numbers of semi-simple Lie groups with finite center (Corollary \ref{ssgroups}), simple algebraic groups over local fields (Corollary\ref{algebraic}), and automorphism groups of locally finite trees acting transitively on the boundary (Corollary \ref{trees}).

\subsection*{Notation and conventions}
We denote by $\bar{\mathbb{N}}_0$ the set $\mathbb{N}\cup \{0\}Ê\cup \{\infty\}$ and for extended-real numbers adopt the usual convention in stating equalities that $0\cdot \infty=0$.

Countable sets are those that embed in the natural numbers, i.e.~this designation includes finite sets.

\subsection*{Acknowledgements}
We are very grateful to Pierre-Emmanuel Caprace, Damien Gaboriau, Claudio Nebbia and Stefaan Vaes for useful conversations and/or comments on drafts of this paper.

\section*{Preliminaries}
Let $G$ be a lcus group. For all $n\in \mathbb{N}_0$ the (continuous) cohomology $H^n(G,E)$ of $G$ with coefficients in a complete locally convex topological vector space $E$ is the $n$'th cohomology, $H^n(G,E) := \ker \partial^n / \operatorname{im} \partial^{n-1}$, of the complex
\begin{displaymath}
\xymatrix{ 0 \ar[r] & E \ar[r]^<<<<{\partial^0} & L^2_{loc}(G,E) \ar[r]^<<<<{\partial^1} & \cdots }
\end{displaymath}
where $L^2_{loc}(G^n,E)$ is the locally convex space of locally square integrable functions $f\colon G^n \rightarrow E$, in the sense that for every compact subset $K$ of $G^n$ the restriction $f\vert_K$ is square integrable with respect to the Haar measure on $G$ and every continuous semi-norm on $E$. The spaces $L^2_{loc}$ are endowed with the projective topologies induced by restriction maps to compact subsets and the cohomologies then inherit canonically vector topologies, which we emphasize need not be Hausdorff. For extensive background on continuous cohomology we refer to the monographs \cite{BW80,Gu80}.

We denote by $\lambda, \rho$ the left- respectively right-regular representations of $G$ on $L^2G$, and by $LG$ respectively $RG$ the von Neumann algebras $LG:= \lambda(G)''$ and $RG:=\rho(G)''$. Recall that $LG$ and $RG$ are anti-isomorphic, and so we will generally prefer to think of $L^2G$ as a right-$LG$-module instead of a left-$RG$-module. The group von Neumann algebra $LG$ is semi-finite and carries a canonical (given a fixed scaling of the Haar measure) semi-finite, faithful, normal tracial weight $\psi$ (see e.g.~\cite[Theorem 7.2.7]{Gert}), given on the positive part of $LG$ by $\psi(x^*x)=\int_G |f(g)|^2\;dg$, where $x$ is left convolution by $f\in L^2G$.

For any right-$LG$-module $E$, in the algebraic sense that we require just $E$ to be a module over the ring $LG$ with no assumptions involving the topology on $LG$, there is a notion of $LG$-dimension $\dim_{(LG,\psi)} E$ which in particular extends the von Neumann dimension when $E$ is a closed invariant subspace of $\ell^2(\mathbb{N})\bar{\otimes}ÊL^2G$. For details on the dimension function, which is a semi-finite generalization of the dimension function defined by L{\"u}ck \cite{Lu98I,Lu98II}, we refer to \cite[Appendix B]{Pe11} or for a streamlined approach to \cite[Appendix A]{KPV13}.

In particular, the right-$LG$-module structure on $L^2G$ induces naturally a right-$LG$-module structure on $H^n(G,L^2G)$ for all $n$ and the $L^2$-Betti numbers of $G$ are defined \cite{Pe11} by
\begin{equation}
\beta^n_{(2)}(G) := \dim_{(LG,\psi)} H^n(G,L^2G). \nonumber
\end{equation}

This coincides with previous definitions \cite{CG86,Lu02} in the case where $G$ is countable discrete. We remark also that the $\beta^n_{(2)}(G)$ depend, through the dimension function, on the choice of scaling of the Haar measure on $G$, though we generally suppress this in the notation.

As we noted above the cohomologies $H^n(G,L^2G)$ need not be Hausdorff in general - in fact for $n=1$ and $G$ non-compact, the cohomology is Hausdorff if and only if $G$ is non-amenable - and we will consider in general also the \emph{reduced $L^2$-cohomology}, denoted $\underline{H}^n(G,E)$, which is by definition the (maximal) Hausdorff quotient of $H^n(G,E)$ for any complete, locally convex topological vector space $E$. It is a key result of \cite{KPV13} that in fact
\begin{equation}
\beta^n_{(2)}(G) = \dim_{(LG,\psi)} \underline{H}^n(G,L^2G). \nonumber
\end{equation}

We refer to \cite[Chapter 18]{Dix69} for background on the disintegration theory for the regular representation of plcus groups. The following lemma, describing how to compute the $LG$-dimensions of certain invariant subspaces of $L^2G$, follows directly from \cite[theorem 18.8.1]{Dix69}.

\begin{lemma} \label{lma:formal_dimension}
Let $G$ be a plcus group and fix some Haar measure and corresponding Plancherel measure $\mu$ on $\hat{G}$. Let $\Omega \subseteq \hat{G}$ be a measurable subset and choose a cross section $\xi \colon \Omega \rightarrow \mathcal{H}_{\omega}$ of unit vectors. Then 
\begin{equation}
\int_{\Omega}^{\oplus} \mathcal{H}_{\bar{\omega}} \mathrm{d}\mu(\omega) \cong \int_{\Omega}^{\oplus} \xi(\omega) \otimes \mathcal{H}_{\bar{\omega}} \mathrm{d}\mu(\omega) \subseteq L^2G \nonumber
\end{equation}
is canonically a right-$LG$-module and
\begin{equation}
\dim_{(LG,\psi)}\int_{\Omega}^{\oplus} \mathcal{H}_{\bar{\omega}} \mathrm{d}\mu(\omega) = \mu(\Omega). \nonumber
\end{equation}
\phantom{a} \hfill \qedsymbol
\end{lemma}

\section*{Proof of Theorem \ref{thm:main_compute}}

We fix $n\geq 0$, a plcus group $G$, and a Haar measure on $G$. As in the statement $\mu$ denotes the corresponding Plancherel measure on $\hat{G}$. We fix also a field of representatives $\omega \mapsto (\pi_{\omega},\mathcal{H}_{\omega})$ etc., as in the statement of \cite[Theorem 18.8.1]{Dix69}.

Included in the statement of the Theorem \ref{thm:main_compute} is of course that the function
\begin{equation}
\omega \mapsto \dim_{\mathbb{C}} \underline{H}^n(G,\mathcal{H}_{\omega}) \nonumber
\end{equation}
is measurable, so that the integral is well-defined. 

\begin{lemma} \label{lma:cohom_cross_section} Fix $n\geq 1$.
There is a family $(c_i')_{i\in \mathbb{N}}$ of weakly measurable maps $c_i' \colon \hat{G}\rightarrow L^2_{loc}(G^n,\mathcal{H}_{\omega})$ such that, for almost every $\omega$:
\begin{itemize}
\item the map $c_i'(\omega)$ is a cocycle representing a class $[c_i'(\omega)] \in \underline{H}^n(G,\mathcal{H}_{\omega})$
\item the set of classes $[c_i'(\omega)]_{1\leq i\leq \dim_{\mathbb{C}}\underline{H}^n(G,\mathcal{H}_{\omega})}$ is linearly independent with dense linear span in $ \underline{H}^n(G,\mathcal{H}_{\omega})$. 
\end{itemize}
\end{lemma}

\begin{proof}
Let $S^l_{(m)}\subseteq G$ be compact subsets with non-empty interior such that for all $l\in \mathbb{N}$ the sequence $(S^l_{(m)})_{m\in \mathbb{N}}$ of subsets of $G$ is increasing and $\cup_m S^l_{(m)} = G$, and further such that for all $l,m\in \mathbb{N}$ we have $(S^{l+1}_{(m)})^2 \subseteq S^l_{(m)}$. It is not hard to see that such sets exist. Denote for all $l$ the direct product $K^l_{(m)} := \prod_{i=1}^l S^i_{(m)}$. Then we have for all $m$ a commuting diagram
\begin{displaymath}
\xymatrix{  \varprojlim_{m} \int_{\hat{G}}^{\oplus} L^2(K^{n}_{(m)},\mathcal{H}_{\omega}) \mathrm{d}\mu(\omega) \ar[r]^{\partial^{n}_{(*)}} \ar[d]_{\kappa_m} & \varprojlim_{m} \int_{\hat{G}}^{\oplus} L^2(K^{n+1}_{(m)},\mathcal{H}_{\omega}) \mathrm{d}\mu(\omega) \ar[d]^{\kappa_m} \\ \int_{\hat{G}}^{\oplus} L^2(K^{n}_{(m)},\mathcal{H}_{\omega}) \mathrm{d}\mu(\omega) \ar[r]^{\partial^{n}_{(m)}} & \int_{\hat{G}}^{\oplus} L^2(K^{n+1}_{(m)},\mathcal{H}_{\omega}) \mathrm{d}\mu(\omega) }
\end{displaymath}
where the coboundary maps $\partial^n_{(m)}$ are the usual ones, acting pointwise on the $ L^2(K^{n}_{(m)},\mathcal{H}_{\omega})$. Observe that the pointwise operator norm of the coboundary maps depend only on the measure of the $K^n_{(m)}$ and the degree $n$ whence the direct integrals $\partial^n_{(m)}$ are well-defined for all $m$.

The projective limits are separable, complete metrizable spaces, and it is clear that if we choose $(c_i)_i \subseteq \ker \partial^n_{(*)}$ to have dense linear span, then these correspond to measurable sections $c_i \colon \hat{G} \rightarrow L^2_{loc}(G^n,\mathcal{H}_{\omega})$ consisting of cocycles pointwise almost everywhere. 

To arrange these as claimed in the statement, observe that for every $\omega$ the reduced cohomology $\underline{H}^n(G,\mathcal{H}_{\omega})$ is in natural duality with the continuous homology $H_n(G,\mathcal{H}_{\omega})$, induced by a duality with the space of cycles, and that this separates points on $\underline{H}^n(G,\mathcal{H}_{\omega})$. Hence we may proceed as above to obtain \emph{countably} many measurable cross sections of cycles that a.e.~span the homologies densely. Then it is obvious how to proceed.
\end{proof}

\begin{remark}
The final paragraph of the preceding proof is the only place where we must consider the \emph{reduced} cohomology with coefficients in the $\mathcal{H}_{\omega}$.
\end{remark}

In particular this proves that the function 
\begin{equation}
d:\hat{G}\rightarrow \bar{\mathbb{N}}_0: \omega \mapsto \dim_{\mathbb{C}} \underline{H}^n(G,\mathcal{H}_{\omega}), \quad \omega \in \hat{G} \nonumber
\end{equation}
is measurable. For $k\in \bar{\mathbb{N}}_0$, set $\Omega_k=d^{-1}(k)$.
Thus $\hat{G}$ is the disjoint union of the $\Omega_k$'s, each of these being measurable, and the claim of the theorem is that
\begin{equation}
\beta^n_{(2)}(G) = \sum_{k\in \bar{\mathbb{N}}_0} k\cdot \mu(\Omega_k). \nonumber
\end{equation}

Further, to each set $\Omega_k$ corresponds a central projection $Q_k\in LG\cap RG$ such that $Q_k(L^2G) \cong \int_{\Omega_k}^{\oplus} \mathcal{H}_{\omega}Ê\bar{\otimes} \mathcal{H}_{\bar{\omega}} \mathrm{d}\mu(\omega)$. In particular the $Q_k$ are pairwise orthogonal and sum to the identity. Hence by \cite[Proposition 3.10]{Pe11}, to prove the theorem it is sufficient to show that for all $k\in \bar{\mathbb{N}}_0$ we have
\begin{equation}
\dim_{(LG,\psi)} \underline{H}^n(G,Q_k(L^2G)) = k\cdot \mu(\Omega_k). \nonumber
\end{equation}

Fix such a $k$. We may re-order the restrictions $c_i'\vert_{\Omega_k}$ to obtain (countably many) representatives $\alpha_i\colon \Omega_k \rightarrow L^2_{loc}(G^n,\mathcal{H}_{\omega})$ such that these are (a.e.) cocycles, the induced equivalence classes densely span the cohomology $\underline{H}^n(G,\mathcal{H}_{\omega})$, and are linearly independent (in particular, they are all non-zero a.e.). For any measurable $\Omega \subseteq \Omega_k$ and any weakly measurable section of cocycles $\alpha \colon \Omega \rightarrow  Z^n(G,\mathcal{H}_{\omega})$ which is locally square integrable, in the sense that
\begin{equation}
\forall K\subseteq G^n \: \: \mathrm{ compact } : \int_{\Omega \times K}Ê\lVert\alpha(\omega)(k)\rVert^2_{\mathcal{H}_{\omega}} \mathrm{d}(\mu\times Haar)(\omega,k) < \infty, \nonumber
\end{equation}
we define a right-$LG$-submodule $E_{\alpha}$ of $L^2_{loc}(G^n,Q(L^2G))$, where $Q$ is the central projection in $LG\cap RG$ corresponding to $\Omega$, as the image of the map:
\begin{equation}
\int_{\Omega}^{\oplus}\mathcal{H}_{\bar{\omega}} \mathrm{d}\mu(\omega)\rightarrow L^2_{loc}(G^n,Q(L^2G)):\xi\mapsto [g\mapsto \alpha(g)\otimes\xi]. \nonumber
\end{equation}

It is clear that for every such $\alpha:\Omega\rightarrow Z^n(G,\mathcal{H}_{\omega})$ we have
\begin{enumerate}[(i)]
\item $E_{\alpha} \subseteq Z^n(G,Q(L^2G))$,
\item $\dim_{(LG,\psi)} E_{\alpha} = \dim_{(LG,\psi)} \overline{E_{\alpha}} = \mu(\Omega)$ if $\alpha$ is a.e. non-zero. Indeed the morphism $\int_{\Omega}^{\oplus}\mathcal{H}_{\bar{\omega}}\mathrm{d}\mu(\omega) \rightarrow E_{\alpha}$ is an isomorphism of right-$LG$-modules, and then we can apply \cite[B.31 and B.34]{Pe11} combined with Lemma \ref{lma:formal_dimension}.
\end{enumerate}

Observe that we may consider a decomposition of $\Omega_k$ as a countable disjoint union of measurable subsets $\Omega_k^{(j)}, j\in \mathbb{N}$, each with finite measure $\mu(\Omega_k^{(j)}) < \infty$. Then we can choose, for any $\varepsilon > 0$ measurable sets $\Omega$ such that for all $j$, $\mu(\Omega \cap \Omega_k^{(j)}) > \mu(\Omega_k^{(j)}) - \frac{\varepsilon}{2^j}$ and such that all the $\alpha_i$'s satisfy the integrability condition on $\Omega \cap \left(\cup_{j \in F} \Omega_k^{(j)}\right) \times K$ for every finite set $F$ and every $K\subset G^n$ compact -- observe that this requires explicitly the $\sigma$-compactness of $G$. The point is that, substituting $\Omega \cap \left(\cup_{j \in F} \Omega_k^{(j)}\right)$ for $\Omega$, with $F$ increasing to $\mathbb{N}$ we can arrange that $\Omega$ increases to $\Omega_k$ (in measure) and such that the $\alpha_i$'s all satisfy the integrability condition on $\Omega$.

\begin{lemma} \label{lma:lessthan}
With notation as above, the morphism of right-$LG$-modules $E_{\alpha} \rightarrow \underline{H}^n(G,Q(L^2G))$ is injective for any section $\alpha$ such that $\alpha(\omega) \neq 0 \in \underline{H}^n(G,\mathcal{H}_{\omega})$ for a.e.~$\omega$.
\end{lemma}

\begin{proof}
Suppose by contradiction that the section $\alpha(\cdot)\otimes \xi$ does represent the zero class for some section $\xi\in \int_{\Omega}^{\oplus} \mathcal{H}_{\bar{\omega}} \mathrm{d}\mu(\omega)$. That is, there is a sequence $\eta_m\in L^2_{loc}(G^{n-1},Q(L^2G))$ such that $\partial \eta_m \rightarrow \alpha(\cdot) \otimes \xi$ in $L^2_{loc}(G^n,Q(L^2G))$. Now for all $m$ the $\eta_m$ correspond to sections, which abusively as usual we also denote $\eta_m$ on $\Omega$ and we may write these
\begin{equation}
\eta_m\colon \omega \mapsto \sum_j \eta_j^{(m,\omega)}(\cdot) \otimes e_j^{(\omega)} \in L^2_{loc}(G^{n-1},\mathcal{H}_{\omega}\bar{\otimes}\mathcal{H}_{\bar{\omega}}), \nonumber
\end{equation}
where $\{ e_j\}$ is an orthonormal basis of the direct integral $\int_{\Omega}^{\oplus} \mathcal{H}_{\bar{\omega}} \mathrm{d}\mu(\omega)$. Then the coboundary map just acts as
\begin{equation}
\partial\eta_m\colon \omega \mapsto \sum_j (\partial\eta_j^{(m,\omega)})(\cdot) \otimes e_j^{(\omega)} \in L^2_{loc}(G^{n},\mathcal{H}_{\omega}\bar{\otimes}\mathcal{H}_{\omega}). \nonumber
\end{equation}
Decomposing in the same way $\xi$ as a sum $\xi(\omega) = \sum_j x_j^{(\omega)}e_j^{(\omega)}$ we find a measurable set $\Omega'\subseteq \Omega$ and some $j_0$ such that $\inf_{\omega\in \Omega'} |x_{j_0}^{(\omega)}| > 0$ and $\mu(\Omega')>0$. But clearly $\partial\eta_{j_0}^{(m,\omega)} \rightarrow \alpha(\omega)x_{j_0}^{(\omega)}$ as $m\rightarrow \infty$, for almost every $\omega$, contradicting the hypothesis.
\end{proof}

We conclude, by applying the previous lemma to linear combinations, that for the $\alpha_i$'s fixed above, the (algebraic) direct sum $\oplus_{i=1}^k E_{\alpha_i}$ embeds in $\underline{H}^n(G,Q(L^2G))$ whence
\begin{equation} \label{eq:lessthan}
k\cdot \mu(\Omega) = \sum_{i=1}^k \dim_{(LG,\psi)} E_{\alpha_i} =\dim_{(LG,\psi)}\oplus_{i=1}^k E_{\alpha_i}\leq \dim_{(LG,\psi)} \underline{H}^n(G,Q(L^2G)).
\end{equation}

Since $\Omega\subset\Omega_k$ could be chosen increasing to $\Omega_k$ (in measure), the same holds with $\Omega_k$ in place of $\Omega$. This proves that $LHS\geq RHS$ in Theorem \ref{thm:DS_modification}.

We turn now to the opposite inequality. All notations and every choice made previously remain in force, in particular we keep $k,\Omega,\alpha_i$ as before. Clearly it is sufficient to show that, for $\Omega$ increasing to $\Omega_k$, the opposite inequality holds in \eqref{eq:lessthan}. If $k=\infty$ the statement is clear, so we may suppose that $k\in \mathbb{N}_0$.

We claim that the linear span of the (images of the) $E_{\alpha_i}$ is dense in $\underline{H}^n(G,Q(L^2G))$. Indeed, we may append countably many sections $\beta_m\colon \Omega \rightarrow B^n(G,\mathcal{H}_{\omega})$ such that the $\alpha_i$'s and the $\beta_j$'s together span, for a.e. $\omega$, a dense subspace of $Z^n(G,\mathcal{H}_{\omega})$. Taking any cocycle $c\in Z^n(G,Q(L^2G))$ this corresponds as before to a section
\begin{equation}
c\colon \omega \mapsto \sum_j c_j^{(\omega)} \otimes e_j^{(\omega)} \nonumber
\end{equation}
and the $c_j^{(\omega)}$ are then (a.e.) cocycles in $Z^n(G,\mathcal{H}_{\omega})$. Hence on large subsets of $\Omega$ we may approximate $c(\omega)$ as well as we like by linear combinations of the $\alpha_i$'s and $\beta_j$'s, and the claim follows from this. 

From the claim, we conclude that $\underline{H}^n(G,Q(L^2G))$ contains a dense submodule of dimension $k\cdot \mu(\Omega)$. Then it is easy to see that also $\underline{H}^n(G,Q(L^2G))$ has dimension $k\cdot \mu(\Omega)$. Indeed, write $G^n=\cup_mK_m$ as an increasing union of countably many compact subsets and denote by $Z_{(m)}$ (resp. $B_{(m)}$) the closed image of $Z^n(G,Q(L^2G))$ (resp. $B^n(G,Q(L^2G))$) under restriction to $K_m$ for all $m$. Then $\underline{H}^n(G,Q(L^2G))$ embeds in the projective limit $\varprojlim (Z_{(m)}\ominus B_{(m)})$, which has $LG$-dimension bounded above by $k\cdot \mu(\Omega)$, as follows from \cite[B.31 and B.34]{Pe11}. \hfill \qedsymbol

\begin{remark}
The proof just given can be reformulated more conceptually, as suggested to us by Stefaan Vaes, as follows. Fix for each $\omega \in \hat{G}$ a generating unit vector $\xi_{\omega}$ of $\mathcal{H}_{\omega}$, let $p_{\omega} \in \mathcal{B}(\mathcal{H}_{\omega}\bar{\otimes} \mathcal{H}_{\bar{\omega}})$ be the projection onto $\mathcal{H}_{\omega}\otimes \mathbb{C}\xi_{\omega}$, and let $p:=\int_{\hat{G}}^{\oplus} p_{\omega} \in LG$. Then the corner $pLGp$ is isomorphic to  $L^{\infty}(\hat{G},\mu)$ and we get

\begin{proposition}
Let $G$ be a plcus group with a fixed Haar measure and corresponding Plancherel measure $\mu$ on the unitary dual $\hat{G}$ of $G$. Let $\hat{G}\owns \omega \mapsto \mathcal{H}_{\omega}$ be a measurable field of representatives of equivalence classes of irreducible unitary representations. Then for all $n\geq 0$
\begin{equation}
\dim_{L^{\infty}(\hat{G},\mu)} \underline{H}^n\left(G,\int_{\hat{G}}^{\oplus}\mathcal{H}_{\omega}\mathrm{d}\mu(\omega)\right)  = \int_{\hat{G}} \dim_{\mathbb{C}} \underline{H}^n(G,\mathcal{H}_{\omega}) \mathrm{d}\mu(\omega). \nonumber
\end{equation}
\phantom{a} \hfill \qedsymbol
\end{proposition}

Then, since $\int_{\hat{G}}^{\oplus}\mathcal{H}_{\omega}\mathrm{d}\mu(\omega) \simeq L^2G.p$, it is easy to derive Theorem \ref{thm:main_compute} from this, see \cite[Lemma A.16]{KPV13}.
\end{remark}

\section*{Proof of Theorem \ref{thm:DS_modification}}

Theorem \ref{thm:DS_modification} follows immediately from Theorem \ref{thm:main_compute} and the following result.

\begin{proposition}
Let $G$ be a plcus group and let $\omega$ be a square-integrable representation. Then
\begin{equation}
\dim_{\mathbb{C}}H^n(G,\mathcal{H}_\omega) = \dim_{\mathbb{C}}\underline{H}^n(G,\mathcal{H}_\omega). \nonumber
\end{equation}
\end{proposition}

In light of \cite[Proposition III.2.4]{Gu80} we just need to handle the case where $\dim_{\mathbb{C}}H^n(G,\mathcal{H}_\omega)$ is infinite\footnote{Observe that we know of no example where this happens!}. We retain the notation following Lemma \ref{lma:lessthan} and claim that, just as in that lemma we have an embedding of right-$LG$-modules (with $\otimes$ denoting algebraic tensor product):
\begin{equation}
H^n(G,\mathcal{H}_{\omega})\otimes \mathcal{H}_{\bar{\omega}} \rightarrow H^n(G,\mathcal{H}_{\omega} \bar{\otimes} \mathcal{H}_{\bar{\omega}}). \nonumber
\end{equation}

Indeed, given any $\alpha \in Z^n(G,\mathcal{H}_{\omega})$ and any $\xi\in \mathcal{H}_{\bar{\omega}}$ we claim that $\alpha(\cdot)\otimes \xi = 0 \in H^n(G,\mathcal{H}_{\omega}\bar{\otimes} \mathcal{H}_{\bar{\omega}})$ if and only if $\alpha$ is a coboundary, and this is clear by exactly the same proof as in the lemma: the 'if' part is trivial, and to see the 'only if' part, suppose that there is an $\eta = \sum_i \xi_i\otimes \xi$ such that $\partial\eta = \alpha(\cdot) \otimes \xi$ (note that it is clear a priori that $\eta\in \mathcal{H}_{\omega}\otimes \xi$). Then it is clear that $\alpha(\cdot) = \partial( \sum_i \xi_i )$, which proves the claim.

As before we then choose a sequence $(\alpha_i)_{i\in I}$ of representatives of cocycles, which are linearly independent in $H^n(G,\mathcal{H}_\omega)$. Applying the claim to finite linear combinations of the $\alpha_i$'s we then have an embedding
\begin{equation}
H^n(G,\mathcal{H}_{\omega})\otimes \mathcal{H}_{\bar{\omega}} \simeq \bigoplus_{i\in I} E_{\alpha_i} \hookrightarrow H^n(G,\mathcal{H}_{\omega} \bar{\otimes} \mathcal{H}_{\omega})
\end{equation}
of right-$LG$-modules.

Since $I$ is infinite it follows that
\begin{equation}
\dim_{LG} H^n(G,\mathcal{H}_{\omega} \bar{\otimes} \mathcal{H}_{\omega}) = \infty. \nonumber
\end{equation}

To conclude the proof note that by \cite[Theorem A]{KPV13} it easily follows that the same holds for $\dim_{LG} \underline{H}^n(G,\mathcal{H}_{\omega} \bar{\otimes} \mathcal{H}_{\omega})$ whence
\begin{equation}
\dim_{\mathbb{C}}\underline{H}^n(G,\mathcal{H}_\omega) = \infty \nonumber
\end{equation}
by the proof of Theorem \ref{thm:main_compute}. This finishes the proof of the proposition and thus of the theorem. \hfill \qedsymbol

\section*{Examples}

\subsection*{Semisimple Lie groups}

Let $G$ be a connected, semisimple Lie group with finite center. Recall that, for $\pi\in \hat{G}$, the {\it infinitesimal character} of $G$ is the homomorphism $\alpha_\pi:Z({\mathcal U}(G))\rightarrow \C$ on the center $Z({\mathcal U}(G))$ of the universal enveloping algebra ${\mathcal U}(G)$ of the Lie algebra of $G$, obtained by differentiating $\pi$, extending $d\pi$ to a homomorphism on ${\mathcal U}(G)$ by the universal property, and restricting to the center (Schur's lemma asserting that, by irreducibility, the action of $Z({\mathcal U}(G))$ is scalar). Let also $K$ be a maximal compact subgroup of $G$, we denote by $X=G/K$ the associated Riemannian symmetric space.

\begin{cor}\label{ssgroups} $\beta^n_{(2)}(G)=\left\{
\begin{array}{ccccc }
   0   &  \mbox{if either} & rk(G)\neq rk(K) & \mbox{or} & n\neq\frac{1}{2}\dim(X)  \\
   \sum_{\omega\in\hat{G}_d,\alpha_\omega\equiv 0} \mu(\omega)  &   \mbox{if} &r k(G)=rk(K) & \mbox{and} & n=\frac{1}{2}\dim(X)
\end{array}\right.$
\end{cor}

We observe that, by Remark 2.9 in \cite{Bor85}, in the second case there exists discrete series with trivial infinitesimal character, so $\beta^n_{(2)}(G)>0$.

\begin{proof} By a celebrated result of Borel \cite{Bor63}, the group $G$ admits a torsion-free co-compact lattice $\Gamma$, and $\beta^n_{(2)}(\Gamma)= \operatorname{covol}(\Gamma)\beta^n_{(2)}(G)$. Now $\Gamma$ is quasi-isometric to $X$, and by the quasi-isometry invariance of $L^2$-cohomology (see e.g. Dodziuk \cite{Dod77}), we have $\underline{H}^n_{(2)}(\Gamma)=\underline{H}^n_{(2)}(X)$. Now the latter has been computed as a $G$-module by Borel \cite{Bor85}:
$$\underline{H}^n_{(2)}(X)=\left\{
\begin{array}{ccccc }
   0   &  \mbox{if either} & rk(G)\neq rk(K) & \mbox{or} & n\neq\frac{1}{2}\dim(X)  \\
   \oplus_{\omega\in\hat{G}_d,\alpha_\omega\equiv 0} \omega  &   \mbox{if} &rk(G)=rk(K) & \mbox{and} & n=\frac{1}{2}\dim(X)
   \end{array}\right.$$
   This already takes care of the first case of our result. For the second case, taking $\Gamma$-dimensions we have:
   $$\beta^n_{(2)}(\Gamma)=\dim_\Gamma \underline{H}^n_{(2)}(X)=\sum_{\omega\in\hat{G}_d,\alpha_\omega\equiv 0}\dim_\Gamma \omega.$$
   But by equation (3.3) in Atiyah-Schmid \cite{AS77} (nicely spelled out in section 3.3.d of \cite{GHJ89}), we have, for $\omega$ a square-integrable representation: $\dim_\Gamma \omega=\operatorname{covol}(\Gamma)\mu(\omega)$. Putting everything together, we get the desired formula. Observe that, by Theorem 3.16 in \cite{AS77}, for a suitable normalization of the Haar measure of $G$, formal dimensions $\mu(\omega)$ are given by Weyl's dimension formula, hence are positive integers.
\end{proof}

\begin{example}\label{PSL2} Endow $G=PSL_2(\R)$ with the Haar measure that, when viewed on the Poincar\'e disk, gives area $\pi$ to ideal triangles. Then there are exactly two discrete series representations $\omega^+,\omega^-$ with trivial infinitesimal character: there are described geometrically as the representations of $G$ on holomorphic and anti-holomorphic square-integrable 1-forms on the Poincar\'e disk. By page 148 in \cite{GHJ89}, we have $\mu(\omega^+)=\mu(\omega^-)=\frac{1}{4\pi}$, so $\beta^1_{(2)}(G)=\frac{1}{2\pi}$, and $\beta^n_{(2)}(G)=0$ for $n\neq 1$. (Compare with Example 3.16 in \cite{Pe11}).

Of course the value of $\beta^1_{(2)}(G)$ can also be obtained through the formula $\beta^1_{(2)}(G)=\frac{\beta^1_{(2)}(\Gamma)}{\operatorname{covol}(\Gamma)}$, taking e.g. for $\Gamma$ the fundamental group of a closed surface of genus $g\geq 2$: then $\beta^1_{(2)}(\Gamma)=2g-2$, and $\operatorname{covol}(\Gamma)=2\pi(2g-2)$ by Gauss-Bonnet.
\end{example}

\begin{remark*} Set $\ell_0=rk(G)-rk(K)$ and assume that $\ell_0>0$. In Proposition 2.8 of \cite{Bor85} (and Remark 1 following it), Borel proves a non-vanishing result implying that, for $i\in ]\frac{\dim X-\ell_0}{2},\frac{\dim X+\ell_0}{2}]$, the cohomology $H^i(G,L^2(G))$ is infinite-dimensional, although the corresponding reduced cohomology vanishes. This illustrates the importance of working with reduced cohomology in our Theorem \ref{thm:main_compute}.
\end{remark*}

\subsection*{Simple algebraic groups over local fields}

Let $K$ be a non-Archimedean local field with residue field $k$, and let $\mathbb{G}$ be a simple algebraic group defined over $K$. We shall consider the group $G=:\mathbb{G}(K)$ of $K$-points of $\mathbb{G}$. Let $X$ be the Bruhat-Tits building of $G$; set $\ell=\dim(X)$, so that $\ell$ is the $K$-rank of $\mathbb{G}$. 

Let $St$ be the {\it Steinberg module} of $G$, i.e. the representation of $G$ on the $L^2$-cohomology $\underline{H}^\ell_{(2)}(X)$. It follows from 4.10 and 6.2 in \cite{Bor76} (see also Theorem 6.2 in \cite{DJ02}), that $St$ is an irreducible representation of $G$, hence $St\in\hat{G}_d$.

\begin{cor}\label{algebraic} With notations as above, assume that the residue field is large enough (i.e. $|k| > \frac{1724^\ell}{25}$); then 
$\beta^n_{(2)}(G)=\left\{
\begin{array}{ccc }
   0   &  \mbox{if} & n\neq\ell;  \\
   \mu(St) &   \mbox{if} &n=\ell.
\end{array}\right.$
\end{cor}

\begin{proof} We appeal to a series of results of Dymara-Januszkiewicz \cite{DJ02}. Assume first $n\neq\ell$. Taking into account Proposition 1.7 of \cite{DJ02} (which deals with the assumption that $|k|$ is large enough), we may appeal to Theorem E of \cite{DJ02}: for $n\neq\ell$, we have $H^n(G,\pi)=0$ for every unitary representation of $G$, in particular for the regular representation on $L^2(G)$. Now assume that $n=\ell$. By Corollary I in \cite{DJ02}, $H^\ell(G,L^2(G))$ identifies as a $G$-module with $\underline{H}^\ell_{(2)}(X)$, which is just $St$ by the remarks preceding the Corollary.
\end{proof}

\begin{remark*} If Haar measure of $G$ is normalized in such a way that the stabilizer in $G$ of a chamber in $X$ has measure 1, then there is an explicit formula for $\mu(St)$ in terms of the generating function of the affine Weyl group of $G$, see Theorem 3.1 in \cite{Dym03}.
\end{remark*}

\subsection*{Groups acting on trees}

Let $X$ be a locally finite tree. We will deal with the class of groups considered in \cite{Ols77,FTN91}, namely closed, non-compact subgroups $G$ of $Aut(X)$ that act transitively on the boundary $\partial X$. By Proposition 10.2 in Chapter I of \cite{FTN91}, the group $G$ has at most two orbits on vertices of $X$, i.e. $X$ is regular or bi-regular. If $G$ has only one orbit on vertices, then $G$ has an open subgroup of index 2 that has two orbits, so we will restrict to that case, and assume that $X$ is the $(k,\ell)$-bi-regular tree (with the proviso that, if $k=\ell$, then $G$ preserves a bi-partition of $X$). Since $G$ acts transitively on edges, we normalize Haar measure so that edge stabilizers have measure 1. 

\begin{proposition}\label{trees} $\beta^1_{(2)}(G)=\frac{k\ell - k-\ell}{k\ell}$ and $\beta^n_{(2)}(G)=0$ for $n\neq 1$.
\end{proposition}

\begin{proof} It is known (see e.g. Theorem 4.7 in \cite{BK90}) that $G$ contains a cocompact lattice $\Gamma$ which is a free group, say $\Gamma\simeq\mathbb{F}_n$. Let $\Gamma\backslash X=(V,E)$ be the quotient graph; this is a bipartite graph with fundamental group $\Gamma$, we denote by $V_k$ (resp. $V_\ell$) the set of vertices of degree $k$ (resp. degree $\ell$). Looking at the bipartition we get $|E|=k.|V_k|=\ell.|V_\ell |$. Now $\Gamma\backslash X$ has the homotopy type of a bouquet of $n$ circles, so computing in two ways its Euler characteristic we get $1-n=|V|-|E|=|V_k|+|V_\ell|-|E|=|E|(\frac{1}{k}+\frac{1}{\ell}-1)$, hence (since $\Gamma$ acts freely on $X$):
$$\operatorname{covol}(\Gamma)=|E|=\frac{(n-1)k\ell}{k\ell -k-\ell}$$
hence
$$\beta^1_{(2)}(G)=\frac{\beta^1_{(2)}(\Gamma)}{\operatorname{covol}(\Gamma)}=\frac{n-1}{\operatorname{covol}(\Gamma)}=\frac{k\ell - k-\ell}{k\ell},$$
and $\beta^n_{(2)}(G)=\frac{\beta^n_{(2)}(\Gamma)}{\operatorname{covol}(\Gamma)}=0$ for $n\neq 1$.

\end{proof}

\begin{example} For $G=Aut(X)$, we could take the free product $\Gamma=\mathbb{Z}/k\mathbb{Z}\ast\mathbb{Z}/\ell\mathbb{Z}$ as a cocompact lattice with covolume 1 in $G$. Since $\beta^1_{(2)}(\Gamma)=\frac{k\ell - k-\ell}{k\ell}$ (see section 4 in \cite{CG86}), in this special case we get another proof of the same formula.
\end{example}

For $G$ in our class of groups,  the dual $\hat{G}$ was described by Ol'shanskii \cite{Ols77} (see also Chapter III in \cite{FTN91}). However it is still an open question whether $G$ is postliminal\footnote{This is however known either if $G=Aut(X)$ (see \cite{Ols75}) or if $G$ is a rank 1 simple algebraic group over a non-archimedean field (see \cite{Ber74}).}. Nevertheless Ol'shanskii's description allowed Nebbia \cite{Neb12} to prove that  there exists a unique representation $\sigma\in{\hat G}$ with $\dim_\mathbb{C} H^1(G,\sigma)=1$, while $H^1(G,\omega)=0$ for all $\omega\in{\hat G}\backslash\{\sigma\}$; it turns out that $\sigma$ lies in the discrete series of $G$. Combining Nebbia's result with Proposition \ref{trees} and our main result, we can compute the formal dimension of $\sigma$:

\begin{cor} If $G$ is postliminal, then $\mu(\sigma)= \frac{k\ell - k-\ell}{k\ell}$.
\hfill $\square$
\end{cor}

When $X$ is the $k$-regular tree and $G$ has two orbits on vertices, we recover the formula $\mu(\sigma)=\frac{k-2}{k}$ from the Remark following Theorem 2.6 in Chapter III of \cite{FTN91}.




\end{document}